\documentclass[11pt]{amsart}

\usepackage{amsmath,enumitem,amssymb,amsthm,mathtools,array}
\usepackage[colorlinks=true,linkcolor=blue,citecolor=blue,urlcolor=blue]{hyperref}
\newtheorem{theorem}{Theorem}[section]
\newtheorem{proposition}[theorem]{Proposition}
\newtheorem{lemma}[theorem]{Lemma}
\newtheorem{corollary}[theorem]{Corollary}
\newtheorem{definition}[theorem]{Definition}
\newtheorem{remark}[theorem]{Remark}

\newcommand{\C}{\mathbb C}
\newcommand{\E}{\mathbb E}

\newcommand{\N}{\mathbb N}
\newcommand{\Nzero}{\mathbb N_0}
\newcommand{\Q}{\mathbb Q}
\newcommand{\R}{\mathbb R}
\newcommand{\T}{\mathbb T}
\newcommand{\Z}{\mathbb Z}
\newcommand{\dd}{\overset{\mathrm{d}}{=}}

\newcommand{\bbS}{\mathbb S}

\DeclarePairedDelimiter{\abs}{\lvert}{\rvert}

\title[Decorated \(p\)-adic sssi processes]{Decorated stable $p$-adic self-similar processes with stationary increments}
\author{Yi Shen}
\author{Zhenyuan Zhang}
\date{}

\begin{document}

\begin{abstract}
We construct new classes of examples of self-similar
processes with stationary increments indexed by $\mathbb Q_p$ via stable integrals. Classical constructions arise from the real counterpart and from discounted branching random walks. We discuss a new decoration technique that significantly enlarges these classes. The decoration technique makes use of the special symmetry of $\mathbb{Q}_p$ to obtain self-similarity and stationarity of increments, and it does not have an analogue on the real line. We also show that these enlarged classes of decorated processes are pairwise incomparable under inclusion.
\end{abstract}

\maketitle

\section{Introduction}

Self-similar processes with stationary increments (sssi processes), introduced by \cite{Lamperti1962}, are defined through probabilistic symmetries that preserve the law under translations and dilations of the time
index $\R$:
\[
  \{X(t+h)-X(h)\}_{t\in\R}\dd\{X(t)\}_{t\in\R},
  \qquad
  \{X(at)\}_{t\in\R}\dd\{a^H X(t)\}_{t\in\R},
\]for all $a>0$ and $h\in\R$, 
where $H>0$. The process can also be indexed by \([0,\infty)\); in that case, the same identities are imposed with \(t\geq 0\) and \(h\geq 0\).

Besides well-known examples such as fractional Brownian motion and stable L\'{e}vy processes, stable-integral constructions provide further explicit classes of sssi processes such as the linear fractional stable motion  \cite{CambanisMaejima1989} and the log-fractional stable motion  \cite{KasaharaMaejimaVervaat1988}; see \cite{SamorodnitskyTaqqu1994,KonoMaejima1991} for a detailed account. Here, by a stable-integral construction, we mean a representation of the form
\[
  X(t)=\int_E f_t(x)\,M_\alpha(dx),\qquad t\in \R,
\]
where \(M_\alpha\) is a symmetric \(\alpha\)-stable random measure on a Borel \(\sigma\)-finite
 measure space \((E,\mathcal E,\mu)\), and the kernels
\(f_t\in L^\alpha(E,\mu)\) are chosen so that the resulting process $\{X(t)\}_{t\in\R}$ has the
desired symmetries.
More broadly, stable sssi processes sit at the intersection of heavy tails, ergodic-theoretic flow representations, limit theorems, and
long-range dependence \cite{Samorodnitsky2016LRD,PipirasTaqqu2017LRD,PipirasTaqqu2017StableSSSI}. For a classical introduction to self-similar processes, we refer the reader to \cite{EmbrechtsMaejima2002}. 

This paper concerns the $p$-adic analogue of sssi processes.  For a fixed prime \(p\), a real-valued process
\(X=\{X(t)\}_{t\in\Q_p}\) is \(p\)-adic sssi if
\[
  \{X(at)\}_{t\in\Q_p}\dd \{\abs{a}_p^H X(t)\}_{t\in\Q_p},
  \qquad
  \{X(t+h)-X(h)\}_{t\in\Q_p}\dd \{X(t)\}_{t\in\Q_p}
\]
for all \(a\in\Q_p^\times=\Q_p\setminus\{0\}\) and \(h\in\Q_p\), where $H>0$. An immediate consequence is the stochastic continuity with respect to the $p$-adic topology, compared to the Euclidean topology for classical sssi processes on $\R$, and hence different properties and constructions are expected. For example, previous works indicate that these processes arise naturally from the discrete-time analogue of the sssi property \cite{GefferthEtAl2003,ShenZhang2021}, admit a spectral representation \cite{ShenZhang2021}, and have drastically different sample path behaviors from sssi processes on $\R$ \cite{ShenZhang2026AP}.

The goal of this paper is to construct new examples of $p$-adic sssi processes using stable integrals. We start with the obvious constructions. First, \cite{ShenZhang2021} provides a class of tree-like $p$-adic sssi processes, with an interesting connection to discounted branching random walks \cite{BenjaminiGurelGurevichSolomyak2009} and Bernoulli convolutions \cite{PeresSchlagSolomyak}. Such a class can be easily reformulated using a stable-integral representation.     
Second, many classical stable sssi processes on the real line, such as the linear fractional stable motion, the log-fractional stable motion, and the harmonizable fractional stable motion, have $p$-adic analogues. These constructions will be presented in Section \ref{sec:simple}.

Our main contribution is introducing further non-trivial stable integral constructions that lead to larger classes of \(p\)-adic sssi processes (Section \ref{sec:decorated}). Employing the symmetric structure of $\Q_p$, we develop a \textit{decoration} technique that expands the above classes of examples. Intuitively, decorated processes take slices of (quotient) spaces in $\Q_p$, assign arbitrary weights and random shifts to the stable integrals on these slices, and take a weighted average over these integrals. Their stable-integral representations often involve product spaces of such slices and finite groups, and hence admit extra degrees of freedom compared to the classical constructions, sometimes even in a continuum manner. 
Table \ref{tab:real-padic-decorated} summarizes our main findings. To the best of our knowledge, the decoration technique is new and does not have an analogue on the real line for sssi processes.

\begin{table}[t]
\centering
\footnotesize
\renewcommand{\arraystretch}{1.25}
\begin{tabular}{|>{\raggedright\arraybackslash}p{0.27\textwidth}|
                >{\raggedright\arraybackslash}p{0.30\textwidth}|
                >{\raggedright\arraybackslash}p{0.30\textwidth}|}
\hline
\textbf{Real sssi model} &
\textbf{\(p\)-adic sssi analogue} &
\textbf{Decorated \(p\)-adic extension} \\
\hline
Linear fractional stable motion
\cite[Section 3]{CambanisMaejima1989}
&
\(p\)-adic linear fractional stable motion
(Definition~\ref{def:direct-lfsm})
&
Decorated \(p\)-adic linear fractional stable motion
(Definition~\ref{thm:linear-fractional})
\\
\hline
Log-fractional stable motion
\cite[Section 3]{KasaharaMaejimaVervaat1988}
&
\(p\)-adic log-fractional stable motion
(Definition~\ref{prop:fractional-kernels})
&
N/A (Remark \ref{rem:log+?})
\\
\hline
Harmonizable fractional stable motion
\cite[Section 4]{CambanisMaejima1989}
&
\(p\)-adic harmonizable fractional stable motion
(Definition~\ref{thm:harmonizable})
&
Decorated \(p\)-adic harmonizable fractional stable motion
(Definition~\ref{thm:packet})
\\
\hline
N/A
&
BRW tree process
\cite[Example~4.1]{ShenZhang2021}
&
Decorated BRW tree process
(Definition~\ref{thm:tree-block})
\\
\hline
\end{tabular}

\vspace{0.5em}
\caption{A summary of classical real stable sssi motions, their \(p\)-adic analogues, and
decorated \(p\)-adic extensions.}
\label{tab:real-padic-decorated}
\end{table}

In addition to constructing new classes of processes, this paper also makes contribution in two other aspects. First, applying the case $\alpha=2$ in our construction immediately yields various simple Gaussian-integral representations of the $p$-adic (fractional) Brownian motion, which was first constructed in \cite{BikulovVolovich1997} and later studied and extended in \cite{Kamizono2007,Kamizono2009,KhrennikovEtAl2013} using techniques such as Paley--Wiener representations and wavelet expansions.  Second, we apply comparison tests to distinguish stable-integral constructions. In particular, we show that distinct classes of decorated processes are pairwise incomparable under inclusion in Section \ref{sec:4}. Our comparison tests hinge on the self-similarity index, the support of the projective spectral measure, and the finiteness of the minimal spectral supremum.

\section{Definitions and the stable-integral test}\label{sec:2}

We fix the notation and basic definitions that will be used throughout this paper. Fix a prime \(p\). The Haar measure on \(\Q_p\) is denoted by \(m\), normalized such that $m(\Z_p)=1$. 
We use \(v_p\) to denote the $p$-adic valuation, so \(\abs{x}_p=p^{-v_p(x)}\) for \(x\in\Q_p^\times\).
Throughout, let \(\N=\{1,2,\ldots\}\), \(\Nzero=\{0,1,2,\ldots\}\),
\(\T=\R/\Z\), and
\(\bbS^1=\{z\in\C:\abs{z}=1\}\).  We also write \(\mathbb{Z}_p^\times=\mathbb{Z}_p\setminus p\mathbb{Z}_p\) for the group of \(p\)-adic units.

\begin{definition}
Let \(H>0\).  A real-valued process \(X=\{X(t)\}_{t\in\Q_p}\) is called \(p\)-adic
\(H\)-self-similar with stationary increments, abbreviated \(p\)-adic
\(H\)-sssi, if for every \(a\in\Q_p^\times\) and \(h\in\Q_p\),
\begin{equation}\label{eq:sssi-scale}
  \{X(at)\}_{t\in\Q_p}\dd \{\abs{a}_p^H X(t)\}_{t\in\Q_p},
\end{equation}
and
\begin{equation}\label{eq:sssi-increments}
  \{X(t+h)-X(h)\}_{t\in\Q_p}\dd \{X(t)\}_{t\in\Q_p}.
\end{equation}
\end{definition}

\begin{definition}
A real-valued process \(X=\{X_n\}_{n\in\Nzero}\) is called discrete \(p\)-adic \(H\)-sssi if for all \(a\in\N\) and \(h\in\Nzero\),
\[
  \{X_{an}\}_{n\in\Nzero}\dd \{\abs{a}_p^H X_n\}_{n\in\Nzero},
  \qquad
  \{X_{n+h}-X_h\}_{n\in\Nzero}\dd \{X_n\}_{n\in\Nzero}.
\]
\end{definition}

There exists a correspondence between $p$-adic sssi and discrete $p$-adic sssi processes. On one hand, the discrete skeleton of any $p$-adic sssi process is discrete $p$-adic sssi. On the other hand, the $p$-adic analogue of \cite[Theorem~3.9]{ShenZhang2021} shows that any discrete $p$-adic sssi process can be uniquely extended (in distribution) to a $p$-adic sssi process. See Lemma \ref{lem:skeleton-determines} below for details.

Let \((E,\mathcal E,\mu)\) be Borel $\sigma$-finite. We recall that a real symmetric \(\alpha\)-stable
(\(S\alpha S\)) random measure \(M_\alpha\) with control measure \(\mu\) is an
independently scattered random measure such that, for every measurable set 
\(A\in \mathcal E\) with \(\mu(A)<\infty\), $M_\alpha(A)$ 
is \(S\alpha S\) with scale
\(\mu(A)^{1/\alpha}\).  Equivalently,
\[
  \E \left[e^{i\theta M_\alpha(A)}\right]
  =
  \exp(-|\theta|^\alpha\mu(A)),
  \qquad \theta\in\R .
\]
In this case, for every
\(f\in L^\alpha(E,\mu)\), the stable integral \(\int_E f\,dM_\alpha\) is
well-defined and satisfies
\begin{align}
     \E\left[\exp\left(i\theta\int_E f\,dM_\alpha\right)\right]
  =
  \exp\left(-|\theta|^\alpha\int_E |f(x)|^\alpha\,\mu(dx)\right),
  \qquad \theta\in\R .\label{eq:31}
\end{align}
 See
\cite[Sections~3.3--3.5]{SamorodnitskyTaqqu1994} for a detailed treatment. The following proposition is a counterpart of \cite[Proposition 7.3.6]{SamorodnitskyTaqqu1994}, which helps us identify stable processes that are $p$-adic sssi.

\begin{proposition}\label{prop:stable-test}
Let \(0<\alpha\le2\) and \(H>0\). In the above setting, suppose that \(f_t\in L^\alpha(E,\mu)\) for each \(t\in\Q_p\), and
define
\begin{align}
    X(t)=\int_E f_t(x)\,M_\alpha(dx).\label{eq:xt}
\end{align}
Then
\begin{equation}\label{eq:stable-exponent}
  \E \left[\exp\left(i\sum_{j=1}^r\theta_jX(t_j)\right)\right]
  =
  \exp\left(-\int_E
  \bigg|\sum_{j=1}^r\theta_jf_{t_j}(x)\bigg|^\alpha\mu(dx)\right).
\end{equation}
Moreover, \(X\) is \(p\)-adic \(H\)-sssi if and only if the following two identities
hold for every \(r\ge1\), \(t_1,\ldots,t_r\in\Q_p\),
\(\theta_1,\ldots,\theta_r\in\R\), \(h\in\Q_p\), and \(a\in\Q_p^\times\):
\begin{equation}\label{eq:test-increments}
  \int_E\bigg|\sum_{j=1}^r\theta_j(f_{t_j+h}-f_h)\bigg|^\alpha d\mu
  =
  \int_E\bigg|\sum_{j=1}^r\theta_jf_{t_j}\bigg|^\alpha d\mu
\end{equation}
and
\begin{equation}\label{eq:test-scale}
  \int_E\bigg|\sum_{j=1}^r\theta_jf_{a t_j}\bigg|^\alpha d\mu
  =
  \abs{a}_p^{\alpha H}
  \int_E\bigg|\sum_{j=1}^r\theta_jf_{t_j}\bigg|^\alpha d\mu .
\end{equation}
\end{proposition}

\begin{proof}
\sloppy Formula \eqref{eq:stable-exponent} follows directly from \eqref{eq:31}.  If
\eqref{eq:test-increments} holds, then every real linear combination of
\(
  (X(t_1+h)-X(h),\dots,X(t_r+h)-X(h))
\)
has the same characteristic function as the corresponding linear combination of
\((X(t_1),\dots,X(t_r))\).  This gives \eqref{eq:sssi-increments}.  If
\eqref{eq:test-scale} holds, then the characteristic function of
\((X(at_1),\dots,X(at_r))\) equals that of
\((\abs{a}_p^H X(t_1),\dots,\abs{a}_p^H X(t_r))\), which proves \eqref{eq:sssi-scale}. The converse is similar.
\end{proof}

In Sections \ref{sec:simple} and \ref{sec:decorated} below, we will repeatedly apply Proposition \ref{prop:stable-test} to prove that a given process $\{X(t)\}_{t\in\Q_p}$ constructed via \eqref{eq:xt} is $p$-adic sssi. Such a proof will consist of three independent steps:
\begin{itemize}
    \item proving $f_t\in L^\alpha(E,\mu)$ for well-definedness;
    \item proving \eqref{eq:test-increments} for stationarity of increments;
    \item proving \eqref{eq:test-scale} for $p$-adic self-similarity.
\end{itemize}

\section{Simple examples}\label{sec:simple}

In this section, we introduce simple examples of stable $p$-adic sssi processes that have their real counterparts in existing literature. We then apply Proposition \ref{prop:stable-test} to prove that these examples are indeed $p$-adic sssi.

\subsection{Constructions and basic observations}

The following classes of kernels are the \(p\)-adic analogues of the log-fractional stable
motion, the linear fractional stable motion, and the harmonizable fractional stable motion from the classical real-parameter theory of stable sssi processes \cite{CambanisMaejima1989,KasaharaMaejimaVervaat1988}.

\begin{definition}[\(p\)-adic log-fractional stable motion]
\label{prop:fractional-kernels}
Let \(0<\alpha\le2\), and let \(M_\alpha\) be a real symmetric \(S\alpha S\)
random measure on \(\Q_p\) with Haar control measure. The $p$-adic  log-fractional stable motion is defined as\footnote{Here and later in this paper, the singular points such as $x=0,\,t$ can be assigned arbitrary values since they are Haar-null.}
\[
  X(t)=\int_{\Q_p} \left(\log\abs{t-x}_p-\log\abs{x}_p \right)\,M_\alpha(dx).
\]
\end{definition}

\begin{definition}[\(p\)-adic linear fractional stable motion]\label{def:direct-lfsm}
Let \(0<\alpha\le2\) and \(H>0\) with \(H\ne1/\alpha\), and let \(M_\alpha\) be a real symmetric
\(S\alpha S\) random measure on \(\Q_p\) with Haar control measure.  The \(p\)-adic linear fractional stable motion is defined as
\[
  X(t)=\int_{\Q_p}
  \left(|t-x|_p^{H-1/\alpha}-|x|_p^{H-1/\alpha}\right)M_\alpha(dx),
  \qquad t\in\Q_p .
\]
\end{definition}

Define the additive character \(\chi:\Q_p\to\bbS^1\) by $\chi(x)=1$ for $x\in\Z_p$ and
\[
  \chi(p^{-k}u)=\exp\left(2\pi i\,\frac{u\bmod p^k}{p^k}\right),
  \qquad k\ge1,\ u\in\Z_p^\times .
\]

\begin{definition}[\(p\)-adic harmonizable fractional stable motion]\label{thm:harmonizable}
Let \(0<\alpha\le2\), \(H>0\), and let \(\widetilde M_\alpha\) be a complex
isotropic \(S\alpha S\) random measure on \(\Q_p\) with Haar control measure. The $p$-adic harmonizable fractional stable motion is defined as
\[
  X(t)=\operatorname{Re}\int_{\Q_p}(\chi(t\xi)-1)\abs{\xi}_p^{-H-1/\alpha}\,\widetilde M_\alpha(d\xi).
\]
\end{definition}

\begin{remark}
The discrete skeleton of the \(p\)-adic harmonizable fractional stable motion can also be represented directly using the spectral representation formula for discrete $p$-adic sssi processes in $L^2$ \cite[Theorem 5.1]{ShenZhang2021}, interpreted in the stable sense, with $\{A^{(m)}_\ell\}_{m\in\N,\,0< \ell< p^m,\,p\nmid\ell}$ independent $S\alpha S$, whose scale parameter may depend on $m$.
\end{remark}

\begin{remark}
An interesting observation is that the \(p\)-adic kernels here have wider admissible ranges than their real analogues. Indeed, the real linear fractional stable motion and the real harmonizable fractional stable motion require $H\in(0,1)$ and the log-fractional case requires $\alpha\in(1,2]$, due to integrability issues on the real line. These integrability issues no longer persist in the $p$-adic setting.
\end{remark}

\begin{remark}

Since \(\Q_p\) does not carry a natural linear ordering structure, there is no direct \(p\)-adic analogue of the $S\alpha S$ L\'{e}vy process defined on \([0,\infty)\), whose stable-integral representation is given by
\[
X(t)=\int_0^\infty \mathbf{1}_{\{x\le t\}}\,M_\alpha(dx)
=\int_0^t M_\alpha(dx),\qquad t\ge 0;
\]
see \cite[Example 3.6.1]{SamorodnitskyTaqqu1994}. 
\end{remark}

 The next (stable) BRW tree process is the main  example for discrete \(p\)-adic sssi
processes; see \cite[Example 4.1]{ShenZhang2021}, which is not restricted to stable processes.  It is closely analogous to branching
random walk with exponentially decreasing steps (or discounted branching
random walks) 
\cite{BenjaminiGurelGurevichSolomyak2009,Zhang2026DiscountedBRW}.

\begin{definition}[BRW tree process]\label{def:brw-tree}
Let \(0<\alpha\le2\), \(H>0\), and let
\[
  E=\bigsqcup_{k=0}^{\infty}\Z/p^{k+1}\Z
\]
with control measure \(\mu\) equal to the counting measure on each
level $k$.  Let \(M_\alpha\) be a real symmetric \(S\alpha S\) random measure on
\(E\) with control measure \(\mu\).  Define the (stable) BRW tree process as
\begin{align}\label{eq:brw-tree}
  \begin{split}
      X_n=\int_{E}p^{-kH}
  \left(\mathbf 1_{\{r\equiv n\!\!\!\pmod {p^{k+1}}\}}
        -\mathbf 1_{\{r\equiv0\!\!\!\pmod {p^{k+1}}\}}\right)
  &M_\alpha(dk,dr),\\
  &\quad n\in\N_0.
  \end{split}
\end{align}
\end{definition}

\begin{remark}
       There is yet another class of \(p\)-adic sssi
processes given by \cite[Example~4.5]{ShenZhang2021}, but we do not include it here because its marginals cannot be made
nondegenerate $S\alpha S$.  For instance, its marginal at \(n=1\) has the form $X_1=\sum_{k\ge0}c^k Z_k,$
where $c\in(0,1)$ and the \(Z_k\)'s are i.i.d.~and \(Z_0\) has an atom at \(0\).  Therefore, $X_1 \dd Z_0+cX_1',$
with \(X_1'\) an independent copy of \(X_1\).  If \(X_1\) were
\(S\alpha S\) with characteristic function
\(\mathbb E e^{i\theta X_1}=\exp(-\sigma^\alpha|\theta|^\alpha)\), then
\[
  \mathbb E e^{i\theta Z_0}
  =
  \frac{\exp(-\sigma^\alpha|\theta|^\alpha)}
       {\exp(-\sigma^\alpha c^\alpha|\theta|^\alpha)}
  =
  \exp(-(1-c^\alpha)\sigma^\alpha|\theta|^\alpha).
\]
Thus, \(Z_0\) would itself be nondegenerate \(S\alpha S\), which is impossible
because \(Z_0\) has an atom at \(0\).  Therefore, \(X_1\) cannot be
nondegenerate \(S\alpha S\).
\end{remark}

\subsection{Verifying the $p$-adic sssi property}

In this subsection, we show that the constructions from the previous subsection all lead to $p$-adic sssi processes.

\begin{theorem}\label{thm:basic examples}
The $p$-adic log-fractional stable motion, the $p$-adic linear fractional stable motion, and the $p$-adic  harmonizable fractional stable motion are $p$-adic sssi. Moreover, the BRW tree process is discrete $p$-adic sssi.
\end{theorem}

\begin{proof}
\textbf{$p$-adic log-fractional stable motion.} Define
\begin{equation*}
  f_t^{\log}(x)=\log\abs{t-x}_p-\log\abs{x}_p .
\end{equation*}
We first check the \(L^\alpha\)-integrability of $f_t^{\log}$.  The case \(t=0\) is trivial, so we fix \(t\ne0\) and write \(\abs{t}_p=p^N\).

For \(\abs{x}_p>\abs{t}_p\), the ultrametric inequality gives
\(\abs{t-x}_p=\abs{x}_p\), so the kernel vanishes.  Thus, the only contribution arises from 
\(\{\abs{x}_p\le\abs{t}_p\}\). We further discuss two cases according to whether equality is achieved.

\begin{itemize}
    \item First, consider the region
\(\{\abs{x}_p<\abs{t}_p\}\).  In this case, 
\(\abs{t-x}_p=\abs{t}_p\), so \(
f_t^{\log}(x)=N\log p-\log |x|_p.
\) Let
\begin{align}
     S_j:=\{x\in\mathbb Q_p: |x|_p=p^j\}, \qquad j\in\Z.\label{eq:Sj}
\end{align}
On \(S_j\), \(
|f_t^{\log}(x)|^\alpha
=\bigl((N-j)\log p\bigr)^\alpha
\). Moreover, \(m(S_j)=(1-p^{-1})p^j\).
Thus, the contribution from \(\{\abs{x}_p<\abs{t}_p\}\) is
\[
\begin{aligned}
\int_{\{\abs{x}_p<\abs{t}_p\}} |f_t^{\log}(x)|^\alpha\,m(dx)
&=
\sum_{j=-\infty}^{N-1}
\int_{S_j} |f_t^{\log}(x)|^\alpha\,m(dx) \\
&=
\sum_{j=-\infty}^{N-1}
\bigl((N-j)\log p\bigr)^\alpha m(S_j) \\
&=
(1-p^{-1})
\sum_{j=-\infty}^{N-1}
p^j\bigl((N-j)\log p\bigr)^\alpha<\infty.
\end{aligned}
\]
\item Next, consider the region \(\{\abs{x}_p=\abs{t}_p\}\). On
\(\{\abs{t-x}_p<\abs{t}_p\}\), the change of variable \(x\mapsto t-x\)
shows that the same summable series above controls 
\(\abs{\log\abs{t-x}_p-\log\abs{t}_p}^\alpha\).  On the remaining part of
\(\{\abs{x}_p=\abs{t}_p\}\), both \(\abs{x}_p\) and \(\abs{t-x}_p\) equal
\(\abs{t}_p\), so the log kernel is zero.
\end{itemize}
Combining the two cases leads to \(f_t^{\log}\in L^\alpha(\Q_p,m)\).

The stationarity of increments follows simply from the facts that
\[
  f_{t+h}^{\log}(x)-f_h^{\log}(x)=f_t^{\log}(x-h)
\]
outside a Haar-null set, and that the Haar measure is translation invariant.

For self-similarity, let \(a\in\Q_p^\times\).  Then
\[
  f_{at}^{\log}(ay)
  =
  \log\abs{a(t-y)}_p-\log\abs{ay}_p
  =
  f_t^{\log}(y)
\]
outside a finite Haar-null set.  The change of variable \(x=ay\) leads to, for every \(r\ge1\), \(t_1,\ldots,t_r\in\Q_p\),
\(\theta_1,\ldots,\theta_r\in\R\),
\[
\begin{aligned}
\int_{\mathbb Q_p}\bigg|\sum_{j=1}^r \theta_j f_{a t_j}^{\log}(x)\bigg|^\alpha m(dx)
&=
\int_{\mathbb Q_p}\bigg|\sum_{j=1}^r \theta_j f_{a t_j}^{\log}(ay)\bigg|^\alpha m(d(ay)) \\
&=|a|_p\int_{\mathbb Q_p}\bigg|\sum_{j=1}^r \theta_j f_{t_j}^{\log}(y)\bigg|^\alpha m(dy).
\end{aligned}
\]
Proposition \ref{prop:stable-test} then proves the claim with \(H=1/\alpha\).

\smallskip

\textbf{\(p\)-adic linear fractional stable motion.} Define
\begin{equation*}
  f_t^H(x)
  =
  |t-x|_p^{H-1/\alpha}-|x|_p^{H-1/\alpha}.
\end{equation*}  We first check
\(L^\alpha\)-integrability.  The case \(t=0\) is trivial, so fix \(t\ne0\)
and write \(|t|_p=p^N\).

If \(|x|_p>|t|_p\), then \(|t-x|_p=|x|_p\), so \(f_t^H(x)=0\).  Therefore, only
\(\{|x|_p\le |t|_p\}\) contributes.  On \(\{|x|_p<|t|_p\}\), we have
\(|t-x|_p=|t|_p\), and hence
\[
  |f_t^H(x)|^\alpha
  =
  \left||t|_p^{H-1/\alpha}-|x|_p^{H-1/\alpha}\right|^\alpha
  \le C\left(1+|x|_p^{\alpha H-1}\right)
\]
for a constant \(C=C(t,\alpha,H)>0\).  With
\[
  S_j=\{x\in\Q_p: |x|_p=p^j\},\qquad m(S_j)=(1-p^{-1})p^j,
\]
and hence 
the possible singular contribution near \(0\) is bounded by
\[
  C\sum_{j=-\infty}^{N-1}p^{j(\alpha H-1)}m(S_j)
  =
  C(1-p^{-1})\sum_{j=-\infty}^{N-1}p^{j\alpha H}<\infty.
\]
The region near \(x=t\) is handled in the same way by the change of variable
\(y=t-x\).  On the remaining part of \(\{|x|_p=|t|_p\}\), both
\(|t-x|_p\) and \(|x|_p\) equal \(|t|_p\), so the kernel is zero.  Hence
\(f_t^H\in L^\alpha(\Q_p,m)\).

Stationarity of the increments follows from the fact that 
\(
  f_{t+h}^H(x)-f_h^H(x)=f_t^H(x-h)
\)
outside a Haar-null set, and the translation-invariance of the Haar measure.

For self-similarity, let \(a\in\Q_p^\times\).  Then $f_{at}^H(ay)
  =
  |a|_p^{H-1/\alpha} f_t^H(y).$ 
Therefore, for every \(r\ge1\), \(t_1,\ldots,t_r\in\Q_p\), and
\(\theta_1,\ldots,\theta_r\in\R\),
\[
\begin{aligned}
\int_{\Q_p}\bigg|\sum_{j=1}^r\theta_j f_{a t_j}^H(x)\bigg|^\alpha m(dx)
&=
|a|_p
\int_{\Q_p}\bigg|\sum_{j=1}^r\theta_j f_{a t_j}^H(ay)\bigg|^\alpha m(dy)\\
&=
|a|_p^{\alpha H}
\int_{\Q_p}\bigg|\sum_{j=1}^r\theta_j f_{t_j}^H(y)\bigg|^\alpha m(dy).
\end{aligned}
\]
Proposition~\ref{prop:stable-test} then concludes the proof.

\smallskip

\textbf{$p$-adic harmonizable fractional stable motion.} Let
\begin{equation*}
  F_t(\xi)=(\chi(t\xi)-1)\abs{\xi}_p^{-H-1/\alpha}.
\end{equation*}
For \(t=0\), the kernel is zero.  If \(t\ne0\), then \(\chi(t\xi)-1=0\)
whenever \(\abs{t\xi}_p\le1\).  Choose \(N\in\Z\) such that
\(\abs{t}_p^{-1}=p^N\). Also, note that \(|\chi(t\xi)-1|\le 2\) implies 
\(
|F_t(\xi)|^\alpha\le2^\alpha |\xi|_p^{-\alpha H-1}
\). 
Using the same notation \eqref{eq:Sj}, and recalling that
\(m(S_j)=(1-p^{-1})p^j\), we obtain
\[
\begin{aligned}
\int_{\mathbb Q_p}|F_t(\xi)|^\alpha\,m(d\xi)
&\le
2^\alpha\sum_{j=N+1}^{\infty}
\int_{S_j}|\xi|_p^{-\alpha H-1}\,m(d\xi) \\
&=
2^\alpha\sum_{j=N+1}^{\infty}
p^{-j(\alpha H+1)}m(S_j) \\
&=
2^\alpha(1-p^{-1})
\sum_{j=N+1}^{\infty}p^{-j\alpha H}
<\infty .
\end{aligned}
\]
Thus, \(F_t\in L^\alpha(\mathbb Q_p,m)\).

For stationarity of the increments, note that
\[
  F_{t+h}(\xi)-F_h(\xi)
  =
  \chi(h\xi)(\chi(t\xi)-1)\abs{\xi}_p^{-H-1/\alpha}.
\]
The extra factor \(\chi(h\xi)\) has modulus one, so \eqref{eq:test-increments}
holds.

For self-similarity, set \(\eta=a\xi\).  Then
\[
  F_{at}(\xi)
  =
  \abs{a}_p^{H+1/\alpha}F_t(\eta),
  \qquad
  m(d\xi)=\abs{a}_p^{-1}m(d\eta).
\]
Following a derivation similar to the counterpart in the above cases, this gives \eqref{eq:test-scale}.  Using a complex version of Proposition \ref{prop:stable-test} (see \cite[Theorem 6.3.1]{SamorodnitskyTaqqu1994}) and taking the real part completes the
proof.

\smallskip

\textbf{BRW tree process.} This follows directly from \cite[Proposition 4.2]{ShenZhang2021}.
\end{proof}

\begin{corollary}
    Taking $\alpha=2$ in Definitions \ref{prop:fractional-kernels}, \ref{def:direct-lfsm}, \ref{thm:harmonizable}, and \ref{def:brw-tree} leads to three different Gaussian-integral representations of (a constant multiple of) the $p$-adic fractional Brownian motion. (Note that Definitions \ref{prop:fractional-kernels} and \ref{def:direct-lfsm} together count as one representation, covering the cases \(H=1/\alpha\) and \(H\ne 1/\alpha\), respectively).
\end{corollary}

\section{Decorated processes}\label{sec:decorated}

In this section, we provide generalizations of the examples from Section \ref{sec:simple} with ``decoration'', the exact meaning of which will be clear as the section unfolds. We then apply Proposition \ref{prop:stable-test} to prove the $p$-adic sssi property.

\subsection{Constructions and basic observations}

We start from the decorated $p$-adic linear fractional stable motion. Fix integers \(c,L\ge1\), and denote the uniform probability
measure on \((\Z/L\Z)\times(\Z_p^\times/(1+p^c\Z_p))\) by \(\nu_{L,c}\).
Intuitively, the group \(\Z_p^\times/(1+p^c\Z_p) \simeq(\Z/p^c\Z)^\times\) represents \(p\)-adic unit directions modulo \(p^c\). 
For \(z\ne0\), we define its unit direction modulo \(p^c\) by
\[
  \omega_c(z)=p^{-v_p(z)}z\pmod{1+p^c\Z_p}\in
  \Z_p^\times/(1+p^c\Z_p).
\]
Intuitively, \(\omega_c(z)\) records the unit part of \(z\) up to \(c\) \(p\)-adic
digits.

\begin{definition}[Decorated $p$-adic linear fractional stable motion]\label{thm:linear-fractional}
Let \(0<\alpha\le2\), \(H>0\), and
\(\psi:(\Z/L\Z)\times(\Z_p^\times/(1+p^c\Z_p))\to\R\) be any function.  For
\((s,\gamma)\in(\Z/L\Z)\times(\Z_p^\times/(1+p^c\Z_p))\), define
\[
  G_{s,\gamma}(z)=
  \begin{cases}
    \abs{z}_p^{H-1/\alpha}
    \psi\bigl(s+v_p(z)\!\!\pmod L,\gamma\omega_c(z)\bigr),&z\ne0,\\
    0,&z=0.
  \end{cases}
\]
Let \(M_\alpha\) be a real symmetric \(S\alpha S\) random measure on
\(\Q_p\times(\Z/L\Z)\times(\Z_p^\times/(1+p^c\Z_p))\) with control measure
\(m\otimes\nu_{L,c}\).  The decorated $p$-adic linear fractional stable motion is defined as
\[
  X(t)=\int_{\Q_p\times(\Z/L\Z)\times(\Z_p^\times/(1+p^c\Z_p))}
       \left(G_{s,\gamma}(t-x)-G_{s,\gamma}(-x)\right)\,
       M_\alpha(dx,ds,d\gamma).
\]
\end{definition}

Compared to Definition \ref{def:direct-lfsm}, the decorated processes introduce more degrees of freedom by allowing \(\psi\) to depend on the $p$-adic valuation and unit direction. It is also clear that $\psi=1$ recovers the $p$-adic linear fractional stable motion if $H\neq 1/\alpha$. 
\smallskip

Next, we define the decorated BRW tree processes. 
Let \(r,L\ge1\),
\[
  E=\bigsqcup_{k\in\Z}(\Q_p/p^{k+r}\Z_p)\times(\Z_p/p^r\Z_p)^\times\times(\Z/L\Z),
\]
and equip \(E\) with
\[
  \mu=\sum_{k\in\Z}\#_{\Q_p/p^{k+r}\Z_p}\otimes\nu_{(\Z_p/p^r\Z_p)^\times}\otimes\nu_L,
\]
where \(\#_{\Q_p/p^{k+r}\Z_p}\) is counting measure and
\(\nu_{(\Z_p/p^r\Z_p)^\times}\), \(\nu_L\) are uniform probability measures.
For \(z\in\Q_p/p^{k+r}\Z_p\), if \(x\in z+p^k\Z_p\),
define
\[
  \ell_{k,z}^{(r)}(x)=p^{-k}(x-z)\pmod{p^r\Z_p}\in \Z_p/p^r\Z_p .
\]
Here \(z\) denotes a coset modulo \(p^{k+r}\Z_p\); the ball \(z+p^k\Z_p\) and
the label \(\ell_{k,z}^{(r)}\) are independent of the chosen representative
because \(p^{k+r}\Z_p\subset p^k\Z_p\).
\begin{definition}[Decorated BRW tree process]\label{thm:tree-block}
Let  \(0<\alpha\le2\), \(H>0\), and
\(W:(\Z/L\Z)\times(\Z_p/p^r\Z_p)\to\R\) be any function.

For \(u\in(\Z_p/p^r\Z_p)^\times\) and \(s\in\Z/L\Z\), set
\[
  \Phi_{k,z,u,s}(x)=
  \begin{cases}
    W(s+k\!\!\pmod L,u\ell_{k,z}^{(r)}(x)),&x\in z+p^k\Z_p,\\
    0,&x\notin z+p^k\Z_p.
  \end{cases}
\]
If \(M_\alpha\) is a real
symmetric \(S\alpha S\) random measure on \(E\) with control measure \(\mu\),
then the (stable) decorated BRW tree process is defined as
\begin{align}
     X(t)=\int_E p^{-kH}\bigl(\Phi_{k,z,u,s}(t)-\Phi_{k,z,u,s}(0)\bigr)
       M_\alpha(dk,dz,du,ds).\label{eq:tree kernel}
\end{align}
\end{definition}

\sloppy The map $W$ carries extra degrees of freedom compared to the BRW tree process (Definition \ref{def:brw-tree}), which is contained in the decorated BRW tree class after restriction to
\(\Nzero\).  Indeed, take \(r=L=1\) and $W(\cdot,a)=\mathbf 1_{\{a=0\}}$ in Definition~\ref{thm:tree-block}. 
Then, for every \(u\in(\Z_p/p\Z_p)^\times\),
$\Phi_{k,z,u}(x)
  =
  \mathbf 1_{\{x\equiv z\,\pmod {p^{k+1}\Z_p}\}}$.
For \(n\in\Nzero\), the levels \(k<0\) have no contribution, because
\(n\equiv0\pmod {p^{k+1}\Z_p}\).  At the levels \(k\ge0\), only
\(z\in \Z_p/p^{k+1}\Z_p\) can contribute, and the kernel \eqref{eq:tree kernel} becomes
\[
  p^{-kH}
  \left(
    \mathbf 1_{\{z\equiv n\!\!\!\pmod {p^{k+1}\Z_p}\}}
    -
    \mathbf 1_{\{z\equiv0\!\!\!\pmod {p^{k+1}\Z_p}\}}
  \right),
\]
which is exactly the BRW tree kernel \eqref{eq:brw-tree} in Definition~\ref{def:brw-tree}.

\smallskip

Finally, we define the decorated $p$-adic harmonizable fractional stable motion, which is a discrete-time process. Let $L\geq 1$. 
 For
\(k\ge1\), let \(\nu_{(\Z/p^k\Z)^\times}\) be the uniform probability measure on
\((\Z/p^k\Z)^\times\).  Define
\[
  E=\bigsqcup_{k=1}^\infty \T\times(\Z/p^k\Z)^\times\times(\Z/L\Z)
  ~~\text{ and }~~
  \mu=\sum_{k=1}^\infty
  \lambda_{\T}\otimes\nu_{(\Z/p^k\Z)^\times}\otimes\nu_L,
\]
where \(\lambda_{\T}\) is Haar probability measure on \(\T=\R/\Z\).
\begin{definition}[Decorated $p$-adic harmonizable fractional stable motion]\label{thm:packet}
Let \(0<\alpha\le2\), \(H>0\), and 
\(\phi_s:\T\to\R\), \(s\in\Z/L\Z\) be bounded measurable functions.   For
\(n\in\Nzero\), define
\begin{equation}\label{eq:packet-kernel}
  f_n(k,r,u,s)
  =
  p^{-kH}
  \left(
    \phi_{s+k}\left(r+\frac{un}{p^k}\right)-\phi_{s+k}(r)
  \right),
\end{equation}
where \(s+k\) is modulo \(L\), and \(u n/p^k\) denotes the class in
\(\T\) obtained from any integer representative of
\(u\in(\Z/p^k\Z)^\times\).  Let
\(M_\alpha\) be a real symmetric \(S\alpha S\) random measure on \(E\) with
control measure \(\mu\). Then, the decorated $p$-adic harmonizable fractional stable motion is defined as 
\[
  X_n=\int_E f_n\,dM_\alpha,\qquad n\in\Nzero.
\]
\end{definition}

Again, the choices of $\phi_s,\,s\in \Z/L\Z$ offer extra degrees of freedom. 
To see that the $p$-adic harmonizable fractional stable motion (Definition \ref{thm:harmonizable}) is a special case, take \(L=1\) and set $\phi(r)=(1-p^{-1})^{1/\alpha}\cos(2\pi r).$ 
Then
\begin{align}
     f_n(k,r,u)
  =
  (1-p^{-1})^{1/\alpha}p^{-kH}
  \operatorname{Re}\left(
    e^{2\pi i r}\left(e^{2\pi iun/p^k}-1\right)
  \right).\label{eq:fn}
\end{align}
Now, restrict the $p$-adic harmonizable fractional stable motion to \(n\in\Nzero\).  Its kernel is zero
on \(|\xi|_p\le1\), so only \(|\xi|_p=p^k\), \(k\ge1\) contribute. For $\xi\in S_k$,  
the change of variable \(\xi=p^{-k}u\), \(u\in(\Z/p^k\Z)^\times\) gives
\[
 ( \chi(n\xi)-1)|\xi|_p^{-H-1/\alpha}
  =
  p^{-kH-k/\alpha}(e^{2\pi iun/p^k}-1).
\]
Multiplying by the scaling factor \(m(S_k)^{1/\alpha}=(1-p^{-1})^{1/\alpha}p^{k/\alpha}\) matches  \eqref{eq:fn}, where the factor $e^{2\pi i r}$ arises from rewriting complex isotropic stable noise as real stable noise with a uniform phase $r\in\T$.

\subsection{Verifying the $p$-adic sssi property}

In this section, we show that the constructions in Definitions \ref{thm:linear-fractional}--\ref{thm:packet} all lead to $p$-adic sssi processes.

\begin{theorem}\label{thm:decorated examples}
The decorated $p$-adic linear fractional stable motion and the decorated BRW tree process are $p$-adic sssi. Moreover, the decorated $p$-adic harmonizable fractional stable motion is discrete $p$-adic sssi.
\end{theorem}

\begin{proof}
\textbf{Decorated $p$-adic linear fractional stable motion. }Define
\[
  F_t(x,s,\gamma)=G_{s,\gamma}(t-x)-G_{s,\gamma}(-x).
\]
If \(t=0\), then \(F_t=0\), so suppose that \(t\ne0\).  If
\(\abs{x}_p>p^c\abs{t}_p\), then $1-\frac tx\in1+p^c\Z_p$, and hence \(v_p(t-x)=v_p(-x)\) and \(\omega_c(t-x)=\omega_c(-x)\). It follows that
\begin{equation*}
\abs{x}_p>p^c\abs{t}_p \implies  F_t(x,s,\gamma)=0.
\end{equation*}

Suppose now that  \(\abs{x}_p\leq p^c\abs{t}_p\). Using the bound
\[
  \abs{G_{s,\gamma}(z)}^\alpha
  \le
  \left(\max_{(s,\gamma)\in(\Z/L\Z)\times(\Z_p^\times/(1+p^c\Z_p))}
  \abs{\psi(s,\gamma)}\right)^\alpha\abs{z}_p^{\alpha H-1}
\]
 and the same $S_j$ as defined in \eqref{eq:Sj}, we have for every $N\in\Z$,
\begin{align*}
  \int_{\abs{z}_p\le p^N}\abs{z}_p^{\alpha H-1}\,m(dz)
  &= \sum_{j=-\infty}^N p^{j(\alpha H-1)}m(S_j)=(1-p^{-1})\sum_{j=-\infty}^{N}p^{j\alpha H}<\infty,
  \end{align*}
where we recall  \(m(S_j)=(1-p^{-1})p^j\). It then follows that \(F_t\in L^\alpha\), using a change of variable.

For stationarity of the increments, note that
\begin{equation*}
  F_{t+h}(x,s,\gamma)-F_h(x,s,\gamma)=F_t(x-h,s,\gamma).
\end{equation*}
Translation invariance of \(m\) gives \eqref{eq:test-increments}.

For self-similarity, let \(a\in\Q_p^\times\), and define
\[
  \bar a=p^{-v_p(a)}a\pmod{1+p^c\Z_p}\in
  \Z_p^\times/(1+p^c\Z_p).
\]
For \(z\ne0\), $v_p(az)=v_p(a)+v_p(z)$ and $\omega_c(az)=\bar a\,\omega_c(z),$ which implies $G_{s,\gamma}(az)=\abs{a}_p^{H-\frac{1}{\alpha}}G_{s+v_p(a),\gamma\bar a}(z).$ 
Therefore,
\begin{equation}\label{eq:lfsm-scale-kernel}
\begin{split}
  F_{at}(ay,s,\gamma)
  &=G_{s,\gamma}\left(a(t-y)\right)-G_{s,\gamma}\left(a(-y)\right)\\
&=|a|_p^{H-1/\alpha}\left(G_{s+v_p(a),\gamma\bar a}(t-y)-G_{s+v_p(a),\gamma\bar a}(-y)\right)\\
&=|a|_p^{H-1/\alpha}F_t(y,s+v_p(a),\gamma\bar a).
  \end{split}
\end{equation}
The map \((s,\gamma)\mapsto(s+v_p(a),\gamma\bar a)\) is a measure-preserving permutation of
\((\Z/L\Z)\times(\Z_p^\times/(1+p^c\Z_p))\). 
Since \(m(d(ay))=\abs{a}_p\,m(dy)\),
\eqref{eq:lfsm-scale-kernel} gives \eqref{eq:test-scale} after a change of variables \(x=ay\), similarly as in the proof of Theorem \ref{thm:basic examples}.  Proposition
\ref{prop:stable-test} then proves the theorem.

\smallskip

\textbf{Decorated BRW tree process.} If \(t=0\), the kernel is zero.  Thus, we focus on the case where \(t\neq 0\) below. For
\(k\le v_p(t)-r\), one has \(t\in p^{k+r}\Z_p\subset p^{k}\Z_p\).  Therefore, for
every \(z\in\Q_p/p^{k+r}\Z_p\), \(0\in z+p^k\Z_p\) if and only if \(t\in z+p^k\Z_p\).  In this case,
\[
  \ell_{k,z}^{(r)}(t)-\ell_{k,z}^{(r)}(0)
  =
  p^{-k}t\equiv0\pmod{p^r\Z_p},
\]
so \(\Phi_{k,z,u,s}(t)-\Phi_{k,z,u,s}(0)=0\).  Thus, the level-\(k\)
contribution is zero.

For each fixed level \(k>v_p(t)-r\), \(t\)
belongs to exactly \(p^r\) balls \(z+p^k\Z_p\), because \(z\) is determined
modulo \(p^k\Z_p\) and then has \(p^r\) lifts modulo \(p^{k+r}\Z_p\).  Thus, at
most \(2p^r\) values of \(z\) contribute, and
\begin{align*}
  &\int_E\left|p^{-kH}(\Phi_{k,z,u,s}(t)-\Phi_{k,z,u,s}(0))\right|^\alpha d\mu\\
  &\qquad\le
  2p^r\left(2\max_{(q,a)\in(\Z/L\Z)\times(\Z_p/p^r\Z_p)}
  \abs{W(q,a)}\right)^\alpha
  \sum_{k=v_p(t)-r+1}^\infty p^{-k\alpha H}<\infty ,
\end{align*}
which proves the integrability.

For stationarity of the increments, observe that
\[
  \Phi_{k,z,u,s}(x+h)=\Phi_{k,z-h,u,s}(x).
\]
The map \(z\mapsto z-h\) is a permutation of \(\Q_p/p^{k+r}\Z_p\), so
\eqref{eq:test-increments} follows.

For self-similarity, let \(a\in\Q_p^\times\), and let
\[
  \bar a=p^{-v_p(a)}a\pmod{p^r\Z_p}\in(\Z_p/p^r\Z_p)^\times.
\]
Note that
\[
ax \in z+p^k\Z_p
\Longleftrightarrow
x \in a^{-1}z + p^{k-v_p(a)}\mathbb{Z}_p,
\]
\[
u\,\ell^{(r)}_{k,z}(ax)
=
u\bar a\,\ell^{(r)}_{k-v_p(a),a^{-1}z}(x),
\]
and
\[
(s+v_p(a))+(k-v_p(a)) \equiv s+k \pmod L.
\]
Then
\[
  \Phi_{k,z,u,s}(ax)=
  \Phi_{k-v_p(a),a^{-1}z,u\bar a,s+v_p(a)}(x).
\]
Since
\[
  p^{-kH}=\abs{a}_p^H p^{-(k-v_p(a))H},
\]
and \((k,z,u,s)\mapsto(k-v_p(a),a^{-1}z,u\bar a,s+v_p(a))\) preserves the control
measure, \eqref{eq:test-scale} follows.  Proposition \ref{prop:stable-test}
finishes the proof of this part.

\smallskip

\textbf{Decorated $p$-adic harmonizable fractional stable motion.} Since
\[
  \int_E |f_n|^\alpha\,d\mu
  \le
  \left(2\max_s\|\phi_s\|_\infty\right)^\alpha
  \sum_{k=1}^\infty p^{-k\alpha H}<\infty ,
\]
 the stable integral is well-defined. 

 We now prove stationarity of the increments.  Fix
\(n_1,\ldots,n_m\in\Nzero\), \(\theta_1,\ldots,\theta_m\in\R\), and
\(h\in\Nzero\).  For fixed \(k,u,s\),
\[
  f_{n+h}(k,r,u,s)-f_h(k,r,u,s)
  =
  f_n\left(k,r+\frac{uh}{p^k},u,s\right).
\]
Hence, by translation invariance of Haar measure on \(\T\),
\[
  \int_{\T}\bigg|\sum_{j=1}^m\theta_j
       \{f_{n_j+h}(k,r,u,s)-f_h(k,r,u,s)\}\bigg|^\alpha dr
  =\int_{\T}\bigg|\sum_{j=1}^m\theta_j
       f_{n_j}(k,r,u,s)\bigg|^\alpha dr .
\]
After integrating over \(u,s\) and summing over \(k\), the stable exponent
formula \eqref{eq:stable-exponent} gives stationarity of the increments.

For self-similarity,  first let \(q\in\N\) be coprime to \(p\).  Multiplication by \(q\) is a bijection of
each \((\Z/p^k\Z)^\times\), and $f_{q n}(k,r,u,s)=f_n(k,r,uq,s).$ 
Therefore, \((X_{qn})_{n\ge0}\dd (X_n)_{n\ge0}\). 
It remains to check multiplication by \(p\).  For \(k=1\), \(upn/p=un\) is
zero in \(\T\), so \(f_{pn}(1,r,u,s)=0\).  For \(k\ge2\), let
\(\pi_k:(\Z/p^k\Z)^\times\to(\Z/p^{k-1}\Z)^\times\) be reduction modulo
\(p^{k-1}\).  Then
\[
  f_{pn}(k,r,u,s)
  =
  p^{-H}f_n(k-1,r,\pi_k(u),s+1).
\]
The pushforward of \(\nu_{(\Z/p^k\Z)^\times}\) by \(\pi_k\) is
\(\nu_{(\Z/p^{k-1}\Z)^\times}\).  Summing over \(k\ge2\) and shifting the index therefore gives
\[
  \int_E\bigg|\sum_{j=1}^m\theta_j f_{p n_j}\bigg|^\alpha d\mu
  =
  p^{-\alpha H}
  \int_E\bigg|\sum_{j=1}^m\theta_j f_{n_j}\bigg|^\alpha d\mu .
\]
It follows that \((X_{pn})_{n\ge0}\dd(p^{-H}X_n)_{n\ge0}\).  This finishes the proof.
\end{proof}

\begin{remark}\label{rem:log+?}
In principle, one can add decorations to the logarithmic kernel in Definition \ref{prop:fractional-kernels} by setting
\[
  G_{s,\gamma}(z)
  =
  \lambda\log |z|_p
  +
  \psi\bigl(s+v_p(z)\!\!\!\pmod L,\gamma\omega_c(z)\bigr),
  \qquad z\ne0,~\lambda\in\R,
\]
and then using the kernel $G_{s,\gamma}(t-x)-G_{s,\gamma}(-x)$. Such a decoration is additive instead of multiplicative. 
However, this construction is just a superposition of the $p$-adic log-fractional stable motion and the decorated $p$-adic linear fractional stable motion, because the $\psi$ term is the
decorated linear-fractional kernel at the critical index \(H=1/\alpha\)
in Definition~\ref{thm:linear-fractional}. 
We therefore do not consider it as a new example.
\end{remark}

\section{
Comparison of the constructions}\label{sec:4}

This section compares the constructions in this paper. Our main result is Theorem \ref{thm:four-noncontainment} below, which asserts that there is no strict containment among the classes of  \(p\)-adic log-fractional stable motions, decorated $p$-adic linear fractional stable motions, decorated BRW tree processes, and decorated $p$-adic harmonizable fractional stable motions. Our comparison tests are based on the self-similarity index $H$, the support of the projective spectral measure, and the finiteness of the minimal spectral supremum, corresponding respectively to the first three subsections. The precise definitions will be given below.

\subsection{Elementary tools}

We first collect a few elementary observations that reduce the comparison to the same $H>0$ and to the discrete skeleton, whenever necessary.

\begin{lemma}
\label{lem:skeleton-determines}
Let \(X\) and \(Y\) be \(p\)-adic \(H\)-sssi processes, then \(X\) and \(Y\) have the same finite-dimensional distributions on
\(\Q_p\) if and only if
\begin{align}
    \{X(n)\}_{n\in\Nzero}\dd \{Y(n)\}_{n\in\Nzero}.\label{eq:xnyn}
\end{align}
\end{lemma}

\begin{proof}
We first prove the equivalence in distribution on \(\Z_p\).   That the equality in distribution on \(\Z_p\) implies the equality in distribution on \(\Nzero\) is trivial. For the other direction, since \(\Nzero\) is dense in \(\Z_p\), for each
\(t\in\Z_p\) we can choose integers \(t_j\to t\) in the \(p\)-adic topology.
The $p$-adic sssi property gives
\[
  X(t_j)-X(t)
  \dd
  \abs{t_j-t}_p^H X(1)\overset{\mathrm{p}}{\Rightarrow}0\quad\text{ as }j\to\infty.
\]
The same argument applies to \(Y\).  For finitely many
\(t_1,\ldots,t_m\in\Z_p\), taking integer approximations coordinatewise and
using a union bound gives
\[
  (X(t^{(1)}_j),\ldots,X(t^{(m)}_j))
  \overset{\mathrm{p}}{\Rightarrow}(X(t_1),\ldots,X(t_m))
\] and similarly for \(Y\).  Therefore,  \eqref{eq:xnyn} leads to
\[
  (X(t_1),\ldots,X(t_m))\dd (Y(t_1),\ldots,Y(t_m)).
\]
The equivalence on $\Q_p$ follows directly by scaling using self-similarity.
\end{proof}

The following result is the $p$-adic analogue of \cite[Theorem 1.1]{EmbrechtsMaejima2002}.

\begin{lemma}\label{lem:unique-H}
A nonzero discrete
\(p\)-adic sssi process has a unique self-similarity index.
\end{lemma}

\begin{proof}
Suppose the same nonzero process is both \(H\)-sssi and \(H'\)-sssi.  Then $ X_p\dd p^{-H}X_1\dd p^{-H'}X_1.$ 
The random variable \(X_1\) is not constant nonzero, otherwise the entire process is constant zero by self-similarity. 
As a result, \(H=H'\).
\end{proof}

\subsection{Projective spectral supports}

We first introduce a variant of the spectral measure of \cite[Sections~2.3--2.4]{SamorodnitskyTaqqu1994}, which will serve as a non-trivial comparison test of the processes.

\begin{definition}
Assume \(0<\alpha<2\) and $d\geq 1$.  If a real symmetric \(S\alpha S\) vector
\(V=(V_1,\ldots,V_d)\) has the representation
\[
  V=\int_E f(x)\,M_\alpha(dx),
  \qquad f:E\to\R^d,
\]
then its projective spectral measure is defined as the pushforward of
\(\|f(x)\|^\alpha\mu(dx)\) restricted to \(\{x:f(x)\ne0\}\) under
\[
  x\mapsto [f(x)]\in\mathbf P^{d-1}(\R),
\]
where \(\mathbf P^{d-1}(\R)\) is the \((d-1)\)-dimensional real projective space.
Its projective spectral support is defined as the support of the projective spectral measure. 
\end{definition}

 Our projective spectral measure is the pushforward of the usual
finite-dimensional symmetric stable spectral measure of
\cite[Sections~2.3--2.4]{SamorodnitskyTaqqu1994} under the antipodal quotient
map \(\bbS^{d-1}\to\mathbf P^{d-1}(\R)\). By the uniqueness of the finite-dimensional symmetric stable
spectral measure, the projective spectral support is determined by the law of \(V\); see \cite[Theorem 2.3.1]{SamorodnitskyTaqqu1994}. The next lemma provides the comparison test.

\begin{lemma}\label{lem:projective-tests}
Assume that \(0<\alpha<2\).
\begin{enumerate}[label=(\roman*)]
\item For the \(p\)-adic log-fractional stable motion, the projective spectral support of
      \((X(1),X(p))\) is infinite.\label{test1}
\item If a decorated $p$-adic linear fractional stable motion has \(H=1/\alpha\), then the
      projective spectral support of \((X(1),X(p))\) is finite.\label{test2}
\item Every nonzero decorated $p$-adic linear fractional stable motion with \(H\ne1/\alpha\) has
      some two-dimensional vector with infinite projective spectral support.\label{test3}
\item Every finite-dimensional vector of every decorated BRW tree process has finite
      projective spectral support.\label{test4}
\item There is a decorated $p$-adic harmonizable fractional stable motion with \(H=1/\alpha\) such that the projective
      spectral support of \((X_1,X_p)\) is infinite.\label{test5}

\end{enumerate}
\end{lemma}

\begin{proof}
 \ref{test1} For \(x\in p^m\Z_p^\times\), \(m\ge2\), we have
\[
  (f_1^{\log}(x),f_p^{\log}(x))=(m\log p,(m-1)\log p).
\]
The slopes \((m-1)/m\) are all distinct, and each $p^m\Z_p^\times$ has positive Haar
measure. Therefore, the projective spectral support is infinite.

 \ref{test2} If \(H=1/\alpha\), the factor
\(\abs{z}_p^{H-1/\alpha}\) in Definition \ref{thm:linear-fractional} is
identically one on \(\Q_p^\times\). Moreover, for fixed times \(1\) and \(p\), the values of $G_{s,\gamma}(1-x)$, $G_{s,\gamma}(p-x),$ and $G_{s,\gamma}(-x)$ depend only on finitely many possibilities.  Thus, the vector
\((F_1(x,s,\gamma),F_p(x,s,\gamma))\) takes only finitely many values, so the projective spectral support is finite.

\ref{test3} Suppose that \(H\ne1/\alpha\), and choose
\((s_0,\gamma_0)\) with \(\psi(s_0,\gamma_0)\ne0\).  Fix
\((s,\gamma)\).  Choose \(m_0\pmod L\) so that \(s+m_0\equiv s_0\pmod L\).
Also choose \(b\) such that
\begin{align}
     G_{s,\gamma}(b-1)\neq G_{s,\gamma}(-1).\label{eq:nonconst}
\end{align}
This is possible: if \(G_{s,\gamma}(-1)\ne0\), take $b=1-p^L$; if \(G_{s,\gamma}(-1)=0\), take \(b-1\) such that $s+v_p(b-1)\equiv s_0\pmod L
  $ and $\gamma\omega_c(b-1)=\gamma_0$.

For all sufficiently large $m$ with \(m\equiv m_0\pmod L\), the set
\[
  \{y\in p^m\Z_p^\times:\omega_c(y)=\gamma^{-1}\gamma_0\}
\]
has positive Haar measure.   Since \(y\) is
\(p\)-adically small,
\begin{align*}
    &\bigl(F_1(1-y,s,\gamma),F_b(1-y,s,\gamma)\bigr)
  \\&=
  \left(
    p^{-m(H-1/\alpha)}\psi(s_0,\gamma_0)-G_{s,\gamma}(-1),~
    G_{s,\gamma}(b-1)-G_{s,\gamma}(-1)
  \right).
\end{align*}
As the sufficiently large number \(m\) varies along the infinite progression \(m\equiv m_0\pmod L\), the
first coordinate takes infinitely many values, while the second coordinate is a fixed nonzero number by \eqref{eq:nonconst}.
Therefore, the corresponding two-dimensional projective spectral support is infinite.

\ref{test4} This follows directly from the fact that $W$ is finitely valued in Definition \ref{thm:tree-block}, which implies that for any $t_1,\dots,t_m$,
\(
  \bigl(\Phi_{k,z,u,s}(t_i)-\Phi_{k,z,u,s}(0)\bigr)_{i=1}^m
\)
is finitely valued, and multiplication of $p^{-kH}$ does not change the direction in $\mathbf P^{m-1}(\R)$.

 \ref{test5} Take \(H=1/\alpha\), \,\(L=1\), and
\(\phi_0(r)=\cos(2\pi r)\).  On the level \(k=2\) and the atom
\(u=1\in(\Z/p^2\Z)^\times\),
the kernel vector of \((X_1,X_p)\) is a nonzero
constant multiple of
\[
  \left(
    \cos\left(2\pi r+\frac{2\pi}{p^2}\right)-\cos(2\pi r),\,
    \cos\left(2\pi r+\frac{2\pi}{p}\right)-\cos(2\pi r)
  \right).
\]
The ratio of the second coordinate to the first
coordinate is nonconstant on every interval of \(r\) on which the first coordinate is
nonzero since they have different phases. This implies that the projective spectral support is infinite.
\end{proof}

 \subsection{Minimal spectral-sup functionals}

Let \(0<\alpha<2\), and \(X=(X_n)_{n\in\mathbb N_0}\) be a real symmetric
\(S\alpha S\) process with a spectral representation
\begin{align}
    X_n=\int_E f_n\,dM_\alpha ,\qquad n\in \N_0.\label{eq:repn}
\end{align}
Following \cite{Rosinski1994}, we say that the representation \eqref{eq:repn} is minimal if, after removing the set on which all kernels vanish, the $\sigma$-field generated by \(f_n/f_m\), \(n,m\in \N_0\) is the entire $\sigma$-field modulo null sets. Any real symmetric \(S\alpha S\) process has a minimal spectral representation.

\begin{definition}
Let $X$ be a real symmetric \(S\alpha S\) process with minimal representation given by \eqref{eq:repn}. Then the minimal spectral supremum of $X$ is defined as
\[
    \mathfrak S(f)
    :=
    \int_E \sup_{n\in\mathbb N_0}|f_n(x)|^\alpha\,\mu(dx)
    \in[0,\infty].
\]
\end{definition}

Part \ref{spectralsup1} of the following lemma ensures that the minimal spectral supremum is well-defined, since a minimal spectral representation may not be unique.

\begin{lemma}
\label{lem:minimal-spectral-sup}
Let \(0<\alpha<2\).
\begin{enumerate}[label=(\roman*)]
\item If \(f\) and \(g\) are two minimal spectral representations of the same
process \(X\), then $\mathfrak S(f)=\mathfrak S(g).$ \label{spectralsup1}

\item If, after removing the zero sets, the projective sequence map
\[
    x\longmapsto [(f_n(x))_{n\in\mathbb N_0}]
\]
is one-to-one modulo null sets, then the representation \(f\) is minimal.\label{spectralsup2}

\item If \(X\) has some representation \(f\) with
\(
    \mathfrak S(f)<\infty,
\)
then \(X\) has a minimal representation \(h\) with
\(
    \mathfrak S(h)=\mathfrak S(f)<\infty.
\)
\label{spectralsup3}
\end{enumerate}
\end{lemma}

\begin{proof}
\ref{spectralsup1}  By Rosiński's uniqueness theorem for minimal stable
representations \cite{Rosinski1994}, two minimal representations \(f\) on
\((E,\mu)\) and \(g\) on \((E',\mu')\) of the same $\N_0$-indexed \(S\alpha S\)
process are related, after removing null sets, by an isomorphism
\(\Phi:E'\to E\) and  \(c:E'\to
\mathbb R\setminus\{0\}\) such that
\[
    g_n(y)=c(y)f_n(\Phi (y)),
    \qquad n\in\mathbb N_0,
\]
and
\(    \Phi_\#(|c|^\alpha\mu')=\mu .
\)
It follows that 
\[
\begin{aligned}
    \mathfrak S(g)
    &=
    \int_{E'}\sup_n |g_n(y)|^\alpha\,\mu'(dy)  \\
    &=
    \int_{E'} |c(y)|^\alpha
       \sup_n |f_n(\Phi (y))|^\alpha\,\mu'(dy) \\
    &=
    \int_E \sup_n |f_n(x)|^\alpha\,\mu(dx)
    =
    \mathfrak S(f).
\end{aligned}
\]

\ref{spectralsup2} This is clear from the Lusin--Suslin theorem.

\ref{spectralsup3} Suppose  that \(\mathfrak S(f)<\infty\) for some representation $f$. Let
\(
G(x)=\sup_{n\in \N_0}|f_n(x)|
\)
and we restrict to \(E_0=\{G>0\}\). Let $y(x)=(f_n(x)/G(x))_{n\in \N_0}\in S:=[-1,1]^{\N_0}$ and define the pushforward \(\nu=y_\#(G^\alpha\mu)\). 
With \(g_n(y)=y_n\), \(\nu\) satisfies, for every finite linear combination,
\[
\int_S\bigg|\sum_j\theta_j g_{n_j}(y)\bigg|^\alpha\,\nu(dy)
=
\int_E\bigg|\sum_j\theta_j f_{n_j}(x)\bigg|^\alpha\,\mu(dx).
\]
In other words, \(g\) is an equivalent representation and \(\sup_n|g_n(y)|=1\) on its support. It follows that
\[
  \mathfrak S(g)
  =
  \int_S \sup_{n\in\N_0}|g_n(y)|^\alpha\,\nu(dy)
  =
  \int_S 1\,\nu(dy)
  =
  \nu(S)
  =
  \mathfrak S(f)
  <\infty .
\]

It remains to make the representation $g$ minimal.  Let
\({P}=S/\{\pm1\}\),  \(\pi:S\to P\) be the quotient map, and put
\(\bar\nu=\pi_\#\nu\).  Choose a measurable selector
\(\sigma:P\to S\) and define
\(
  h_n(q)=\sigma(q)_n,\, q\in P .
\)
Since
\(
  |\sum_j\theta_j y_{n_j}|^\alpha
  =
  |\sum_j\theta_j (-y_{n_j})|^\alpha,
\)
the kernels \(h_n\) give the same finite-dimensional laws.  Moreover,
\[
  \mathfrak S(h)
  =
  \int_P \sup_{n\in\N_0}|h_n(q)|^\alpha\,\bar\nu(dq)
  =
  \bar\nu(P)=\nu(S)=\mathfrak S(f)<\infty .
\]
By construction, the map \(q\mapsto[(h_n(q))_{n\in\N_0}]\)
is injective, and hence the representation \(h\) is minimal by \ref{spectralsup2}.
\end{proof}

\begin{lemma}
\label{lem:packet-spectral-sup}
Assume \(0<\alpha<2\).  Then every decorated $p$-adic harmonizable fractional stable motion has finite minimal
spectral supremum.  On the other hand, the \(p\)-adic log-fractional stable motion with $H=1/\alpha$, the $p$-adic linear fractional stable motion with $H<1/\alpha$, and the BRW tree process with $H\leq 1/\alpha$ all have infinite minimal
spectral supremum.
\end{lemma}

\begin{proof}
\textbf{Decorated $p$-adic harmonizable fractional stable motion.} The kernel \eqref{eq:packet-kernel} satisfies
\[
  \sup_{n\in\Nzero}|f_n(k,r,u,s)|
  \le 2\max_s\|\phi_s\|_\infty p^{-kH}.
\]
Therefore,
\[
  \mathfrak S(f)
  \le
  \left(2\max_s\|\phi_s\|_\infty\right)^\alpha
  \sum_{k=1}^{\infty}p^{-k\alpha H}<\infty ,
\]
so it has finite minimal spectral supremum by Lemma~\ref{lem:minimal-spectral-sup}\ref{spectralsup3}.

\vspace{0.2cm}

For each of the three classes below, the representation is minimal after removing
zero sets. Indeed, for the log- and linear-fractional kernels, if
\([(f_n(x))_{n\in\Nzero}]=[(f_n(y))_{n\in\Nzero}]\) and \(x,y\) are distinct elements in $\Z_p$, choosing
\(n_j\in\Nzero\) with \(n_j\to x\) makes \(f_{n_j}(x)\) unbounded while
\(f_{n_j}(y)\) remains bounded, a contradiction.  For the BRW kernel, the zero and nonzero pattern of
\((f_n(k,r))_{n\in\Nzero}\) completely determines \(p^{k+1}\) and \(r\bmod p^{k+1}\),
except when \(p=2,\,k=0\).
This finite duplication can be merged and does not affect the divergence below. Therefore, minimality
follows from Lemma~\ref{lem:minimal-spectral-sup}\ref{spectralsup2}.

\smallskip

\textbf{\(p\)-adic log-fractional stable motion.} 
Since \(\Nzero\) is dense in \(\Z_p\), for every
\(x\in\Z_p\setminus\Nzero\) we can choose \(n_j\in\Nzero\) converging to $x$ in $|\cdot|_p$.  Then
\[
  \bigl|\log|n_j-x|_p-\log|x|_p\bigr|
  \to\infty .
\]
It follows that
\[
  \sup_{n\in\Nzero}\bigl|\log|n-x|_p-\log|x|_p\bigr|=\infty
  \qquad\text{for every }x\in\Z_p\setminus\Nzero .
\]
Since \(m(\Z_p\setminus\Nzero)=1\), the process has infinite minimal spectral supremum.

\smallskip

\textbf{$p$-adic linear fractional stable motion.}  This is similar to the \(p\)-adic log-fractional stable motion, by noting that for $H<1/\alpha$,
\[
  \sup_{n\in\Nzero}\bigl||n-x|_p^{H-1/\alpha}-|x|_p^{H-1/\alpha}\bigr|=\infty
  \qquad\text{for every }x\in\Z_p\setminus\Nzero .
\]

\smallskip

\textbf{BRW tree process.} Recall that the kernel is given by
\[
  p^{-kH}
  \left(
    \mathbf 1_{\{r\equiv n\!\!\pmod {p^{k+1}}\}}
    -
    \mathbf 1_{\{r\equiv0\!\!\pmod {p^{k+1}}\}}
  \right),
  \qquad k\ge0,\ r\in\Z/p^{k+1}\Z .
\]
For every fixed \((k,r)\), choose \(n\) so that exactly one of the two congruences
\[
  r\equiv n\pmod {p^{k+1}}\qquad \text{and}\qquad r\equiv0\pmod {p^{k+1}}
\]
holds.  This is possible: if \(r\equiv0\), take
\(n\not\equiv0\pmod {p^{k+1}}\); if \(r\not\equiv0\), take
\(n\equiv r\pmod {p^{k+1}}\). Therefore,
\[
  \sup_{n\in\Nzero}
  p^{-kH}\left|
    \mathbf 1_{\{r\equiv n\!\!\pmod {p^{k+1}}\}}
    -
    \mathbf 1_{\{r\equiv0\!\!\pmod {p^{k+1}}\}}
  \right|
  \ge p^{-kH}.
\]
Hence
\[
\begin{aligned}
  &\int_E \sup_{n\in\Nzero}\left|
    p^{-kH}
    \left(
      \mathbf 1_{\{r\equiv n\!\!\pmod {p^{k+1}}\}}
      -
      \mathbf 1_{\{r\equiv0\!\!\pmod {p^{k+1}}\}}
    \right)
  \right|^\alpha\,d\mu(k,r)\\
  &\ge
  \sum_{k=0}^\infty
  \sum_{r\in\Z/p^{k+1}\Z} p^{-kH\alpha}  =
  \sum_{k=0}^\infty p^{k+1}p^{-kH\alpha}
  =
  \infty
\end{aligned}
\]
whenever \(H\alpha\le1\).
\end{proof}

\subsection{Main comparison theorem}

In this section, we state and prove the main comparison theorem, using comparison tests based on the previous sections. Note that we will restrict to $\alpha\in(0,2)$, since for the Gaussian case $\alpha=2$, there exists a unique $p$-adic $H$-sssi process for each $H>0$ up to multiplicative constants.

\begin{theorem}
\label{thm:four-noncontainment}
Assume \(0<\alpha<2\). Then none of the following four classes is contained in any of the others: the \(p\)-adic log-fractional stable motion, the decorated $p$-adic linear fractional stable motion, the decorated BRW tree process, and the decorated $p$-adic harmonizable fractional stable motion.
\end{theorem}

\begin{proof}
 By Lemmas \ref{lem:skeleton-determines} and \ref{lem:unique-H}, it remains to compare processes with the same index $H$ and reduce to the discrete skeleton when necessary.

First, we consider the \(p\)-adic log-fractional stable motion, which has index \(H=1/\alpha\). By Lemma \ref{lem:projective-tests}\ref{test1}, the vector \((X(1),X(p))\) has an
infinite projective spectral support, which is not the case for  any decorated $p$-adic linear fractional stable motion with $H=1/\alpha$ and any decorated BRW tree process.  Also, the
log-fractional skeleton has infinite minimal spectral supremum by Lemma
\ref{lem:packet-spectral-sup}, which is not the case for any decorated $p$-adic harmonizable fractional stable motion.  Thus, the log-fractional class is incomparable to each of the
other three classes.

Next, we compare the decorated $p$-adic linear fractional stable motion and the decorated BRW tree process.  For \(H\ne1/\alpha\), any nonzero decorated $p$-adic linear fractional stable motion contains some two-dimensional vector with infinite projective support (Lemma~\ref{lem:projective-tests}\ref{test3}), but any 
two-dimensional vector in a decorated BRW tree process has finite projective support
(Lemma~\ref{lem:projective-tests}\ref{test4}). Thus, the claim follows.

It remains to compare the decorated $p$-adic harmonizable fractional stable motion with the decorated $p$-adic linear fractional stable motion and the decorated BRW tree process.  Choose $X$ in this class such that $H=1/\alpha$ and $(X_1,X_p)$ has infinitely many
projective directions by Lemma \ref{lem:projective-tests}\ref{test5}, then $X$ is not a  decorated $p$-adic linear fractional stable motion or a decorated BRW tree process again by Lemma \ref{lem:projective-tests}\ref{test2} and \ref{test4}.
Conversely, Lemma \ref{lem:packet-spectral-sup} shows that any decorated $p$-adic harmonizable fractional stable motion has finite minimal spectral supremum, while the other two classes all contain examples with infinite minimal spectral suprema. This yields the claim and completes the proof.
\end{proof}

\section*{Acknowledgement}

Yi Shen was supported by the Natural Sciences and Engineering Research Council of Canada [Grant RGPIN-2020-04356]. Zhenyuan Zhang was supported by the Jump Trading Fellowship.

\bibliographystyle{plain}
\bibliography{padic_stable_integral_examples}

\end{document}